\documentclass{article}
\usepackage{graphicx,amsmath,amssymb,color,amsthm} 
\def\forb{{\mathrm{forb}}}
\def\ext{{\mathrm{Ext}}}
\def\Av{{\mathrm{Avoid}}}
\newcommand{\ncols}[1]{\| #1 \|}
\newtheorem{theorem}{Theorem}
\newtheorem{proposition}[theorem]{Proposition}

\newtheorem{defn}[theorem]{Definition}
\newtheorem{conj}[theorem]{Conjecture}

\title{Forbidding $r$-extension of $C_4$ -- a counterexample of the Anstee-Sali conjecture}
\author{Richard P. Anstee\\Department of Mathematics\\University of British Columbia\and  Attila Sali\\HUN-REN Alfréd Rényi Institute of Mathematics and\\Department of Computer Science, BUTE \and Pei Wu\\3dMed Diagnostic Inc. Data Lab}
\date{September 2026}

\begin{document}

\maketitle
\begin{abstract}
    In this note we prove $\mathrm{forb}(n,\mathbf{1}_{r-2}\times C_4)=\Theta(n^{r-1/2})$, thus providing a series of counterexamples of a 2005 conjecture of the first two authors. Here $\mathbf{1}_{r-2}\times C_4$ is the $r$-extension of the 4-cycle, that is the $r$-uniform hypergraph of four edges obtained by joining a set of $r-2$ vertices to each edge of $C_4$. $\mathrm{forb}(n,F)$ is the largest number of edges of a simple hypergraph $G$ on $n$ vertices without \emph{trace} (induced subhypergraph) $F$.
\end{abstract}
\section{Introduction}
An $n\times m$ matrix $A$ is said to be \emph{simple} if it is a (0,1)-matrix with no repeated columns. There is a natural correspondence between columns of $A$ and subsets of $[n]$. We use matrix terminology as follows. Let $\ncols{A}$ be the number of columns of $A$. For a given matrix $F$, we say $F$ is a \emph{configuration} in $A$, denoted $F\prec A$, if there is a submatrix of $A$ which is a row and column permutation of $F$. Note that in the language of set systems a configuration corresponds to a \emph{trace}, or induced subhypergraph. Define 
$$\Av(m,F)=\left\{A\,|\,A \hbox{ is }m\hbox{-rowed and simple}, F\not\prec A\right\},$$ 
$$\forb(m,F)=\max_{A\in\Av(m,F)}\ncols{A}.$$
A matrix $A\in\Av(m,F)$ is called {\emph{extremal}} if $\ncols{A}=\forb(m,F)$ and let 
$$\ext(m,F)=\{A\in\Av(m,F)\,|\,\ncols{A}=\forb(m,F)\}.$$ 
Many results are in a survey \cite{survey}.
A conjecture was formulated in \cite{anstee2005small} on the order of magnitude of $\forb(m,F)$. 
\begin{defn}\label{product}Let $A$ be an $m_1\times n_1$ simple matrix and let $B$ be an $m_2\times n_2$ simple matrix. Then $A\times B$ denotes the 
$(m_1+m_2)\times (n_1n_2)$ simple matrix each column consisting of a column of $A$ placed on a column of $B$ and this is done in all possible ways.\end{defn}

Let $A_i$ be an $n_i\times m_i$ simple matrix
for $1\le i\le t$. The $t$-fold product 
$A=A_1\times A_2\times \cdots \times A_t$ is an
$(\sum_{i=1}^t n_i)\times (\Pi_{i=1}^t m_i)$ simple matrix. Let $I_h$ denote the $h\times h$ identity matrix and $I_h^c$ denotes its (0,1)-complement. Let
$T_h$ denote the
$h\times h$ {\it triangular} matrix, that is 1's on or above the diagonal and 0's below.
Note that if each $A_i$ in the $t$-fold product above is of size $n/t\times n/t$ then the $t$-fold product has $n$ rows and $\Theta(n^t)$ columns. Since we are only interested in the sizes up to a constant factor, we may always assume that $n$ is divisible by the number of terms in the product.
Let $F$ be a $k\times \ell$ (0,1)-matrix.  

\begin{defn}\label{def:XF} Let $X(F)$ be the smallest $p$
so that $F$ is a configuration in
$A_1\times A_2\times \cdots \times A_p$ for every choice of $A_i$ as either
$I_{n/p}$, $I_{n/p}^{c}$ or $T_{n/p}$. Alternatively, assuming $F$ is not a configuration in at least one of $I$, $I^c$, $T$, then $X(F)-1$ is the largest choice of $p$ so that $F$ is not a configuration in
$A_1\times A_2\times \cdots \times A_p$ for some choice of $A_i$ as either
$I_{n/p}$, $I_{n/p}^{c}$ or $T_{n/p}$.\end{defn}
\begin{conj}\label{grand}\cite{anstee2005small} 
$$\forb(n,F)= \Theta(n^{X(F)-1}).$$\end{conj}
In the present paper we prove 
\begin{theorem}\label{thmrc4}
Let $\mathbf{1}_{k}$ denote the $k\times 1$ column of 1's and $C_4$ denote the $4\times 4$ incidence matrix of a four cycle. Then for $r\ge 4$ we have
\[\mathrm{forb}(n,\mathbf{1}_{r-2}\times C_4)=\Theta(n^{r-1/2}.)\]
    \end{theorem}
The lower bound in Theorem~\ref{thmrc4} was found using AI \cite{wu2026counterexample}.
\begin{proposition}\label{prop:XF}
    We have that for $r\ge 4$ \[X(\mathbf{1}_{r-2}\times C_4)=r.\]
\end{proposition}
\begin{proof}
    Since each column of $\mathbf{1}_{r-2}\times C_4$ contains $r$ 1's, it is not contained as a configuration in the product of at most $r-1$ copies of the identity matrix, thus $X(\mathbf{1}_{r-2}\times C_4)\ge r$. 

    On the other hand, consider the product $A_1\times A_2\times \cdots \times A_p$ and assume that one of the $A_i$'s is $I^c$ or $T$. Then one may choose $r-1$ rows  and  two columns of $A_i$ such that they form a matrix $\begin{bmatrix}
        1&1\\1&1\\\vdots&\vdots\\1&1\\0&1
    \end{bmatrix}$. Then we can choose $\begin{bmatrix}
        0&1
    \end{bmatrix}$ from three of the other terms thus obtaining 
    \[\mathbf{1}_{r-2}\times C_4=\mathbf{1}_{r-2}\times \begin{bmatrix}
        0&1
    \end{bmatrix}\times \begin{bmatrix}
        0&1
    \end{bmatrix}\times \begin{bmatrix}
        0&1
    \end{bmatrix}\times \begin{bmatrix}
        0&1
    \end{bmatrix}\prec A_1\times A_2\times \cdots \times A_p.\]
    On the other hand, if each term $A_i$ is an identity matrix, then we may pick  one row from $r-2$ terms with a 1 in it and then two $I_2$'s from the remaining two identity terms. This gives that $\mathbf{1}_{r-2}\times C_4=\mathbf{1}_{r-2}\times I_2\times I_2\prec  A_1\times A_2\times \cdots \times A_p$.
\end{proof}
Thus, Theorem~\ref{thmrc4} and Proposition~\ref{prop:XF} provide a series of counterexamples for Conjecture~\ref{grand}.
\section{Proof of the theorem}
To prove the lower bound we invoke a theorem of Mubayi \cite{MUBAYI_2002}. Let $\mathcal{H}$ be an $r$-uniform hypergraph $ex_r(n,\mathcal{H)}$ denote the largest possible number of edges of a simple $r$-uniform hypergraph $\mathcal{G}$ that does not have a subhypergraph $\mathcal{H}$.
\begin{theorem}\label{thm:Mubayi}\cite{MUBAYI_2002}
   Let $r, t$ be fixed positive integers, $r \ge 2$. Let $\mathcal{H}$ be the complete $r$-partite
$r$-graph $K^{(r)}(1,\ldots , 1, 2, t + 1)$ consisting of $r - 2$ sets of size 1, one set of size 2, and one set
of size $t + 1$. Then
\[ex_r(n, \mathcal{H)}\sim
\frac{\sqrt{t}}{r!}n^{r-1/2}
.\] 
\end{theorem}
Observe that the incidence matrix of $K^{(r)}(1,\ldots , 1, 2, t + 1)$ in case of $t=1$ is exactly $\mathbf{1}_{r-2}\times C_4$, so Theorem~\ref{thm:Mubayi} gives an $r$-uniform hypergraph  on $n$ vertices whose incidence matrix $A$ satisfies $\ncols{A}=\Omega(n^{r-1/2})$ and $\mathbf{1}_{r-2}\times C_4\not\prec A$.

In order to prove the upper bound we apply Standard Induction.
The argument is described in \cite{survey} but we give it here. Let  $\mathcal{F}$ be a family of matrices. We would like to find a bound on $\forb(n,\mathcal{F})$. Choose some $A \in \Av(n,\mathcal{F})$ such that $\|A\| = \forb(n,\mathcal{F})$.

For some row $r$, we can reorder the rows and columns of $A$ to obtain the \emph{standard decomposition} of $A$ on row $r$:

\begin{equation}
A = 
\begin{matrix}
     r\rightarrow \\ \\
\end{matrix}
\begin{bmatrix}
    0 & \cdots & 0 & 1 & \cdots & 1\\
    B_r & & C_r & C_r & & D_r\\
\end{bmatrix},
\label{rdecomp}\end{equation}
If $r$ were deleted, the matrix may not be simple. $C_r$ is the simple matrix of the repeated columns, the columns which appear once with a $0$ in row $r$ and once with a $1$ in row $r$. $B_r$ is the simple matrix of the columns that appear with a $0$ in row $r$, but don't appear with a $1$, and $D_r$ consists of the columns that appear with a $1$ but not a $0$. We note $[B_rC_rD_r]$ is a simple $(m-1)$-rowed matrix avoiding ${\cal F}$, so $\|B_rC_rD_r\| \leq \forb(m-1,\mathcal{F})$. Since $\forb(m,{\cal F}) = \|A\| = \|B_rC_rD_r\| + \|C_r\|$, we obtain:
    \begin{equation}\forb(n,{\cal F}) \leq \forb(n-1,{\cal F}) + \|C_r\|.\label{eq:induction}\end{equation}
If we can find an upper bound on $\ncols{C_r}$ (as a function of $n$), then we have an upper bound on $\forb(n,{\cal F})$ by induction. 

Let $F\in{\cal F}$ and take the standard decomposition of $F$ on some row $t$:
 $$F = 
\begin{matrix}
     t\rightarrow \\ \\
\end{matrix}
\begin{bmatrix}
    0 & \cdots & 0 & 1 & \cdots & 1\\
    B'_t & & C'_t & C'_t & & D'_t\\
\end{bmatrix}.$$ 
For every row $t$, we define $F'_t= [B_t'C_t'D_t']$. If $C_r$ contains $F'_t$, then $A$ contains $F$. We conclude that $C_r \in \Av(n-1,\mathcal{F}')$, where $\mathcal{F}'$ is the set of $F_t'$ obtained for every row $t$ in every matrix $F\in\mathcal{F}$. We call $\mathcal{F}'$ the \emph{inductive children} of $\mathcal{F}$.

We may formalize the induction bound using (\ref{eq:induction} as follows. 
\begin{proposition}\label{prop:ind}
    Let $\mathcal{F}$ be a family of forbidden configurations and let $\mathcal{F}'$ be the inductive children of $\mathcal{F}$. Suppose that $\forb(n,{\cal F}')=O(n^{\alpha})$ Then
    \begin{equation}\label{eq:ind-child}
        \forb(n,{\cal F})=O(n^{\alpha+1}).
    \end{equation}
\end{proposition}
Proposition~\ref{prop:ind} applies to $\mathbf{1}_{r-2}\times C_4$, since $\mathbf{1}_{(r-1)-2}\times C_4$ is one of the inductive children of $\mathcal{F}=\{\mathbf{1}_{r-2}\times C_4\}$. To deal with the base case of the induction we prove
\begin{proposition}\label{prop:r=4}
Let $F_4=\left[\begin{array}{cccc}
1&1&1&1\\ 1&1&1&1\\ 1&0&0&1\\ 1&1&0&0\\ 0&1&1&0\\ 0&0&1&1\\ \end{array}\right]$. Then
    \[
    \forb(n,F_4)=O(n^{\frac{7}{2}}).
    \]
\end{proposition}
\begin{proof}
    Observe that inductive children $\mathcal{F}'$ of $\mathcal{F}=\{F_4\}$ are 
\[
\mathcal{F}'=\{\left[\begin{array}{cccc}
1&1&1&1\\ 1&1&1&1\\ 1&0&0&1\\ 1&1&0&0\\ 0&1&1&0\\  \end{array}\right],
\left[\begin{array}{cccc}
1&1&1&1\\  1&0&0&1\\ 1&1&0&0\\ 0&1&1&0\\ 0&0&1&1\\ \end{array}\right]\}.
\]
    and the inductive children $\mathcal{F}''$ of $\mathcal{F}'$ are
    \[
    \mathcal{F}''=\{\left[\begin{array}{cccc}
  1&0&0&1\\ 1&1&0&0\\ 0&1&1&0\\ 0&0&1&1\\ \end{array}\right],
\left[\begin{array}{cccc}
1&1&1&1\\  1&0&0&1\\ 1&1&0&0\\ 0&1&1&0\\ \end{array}\right],
\left[\begin{array}{cccc}
1&1&1&1\\ 1&1&1&1\\  1&0&0&1\\ 1&1&0&0\\  \end{array}\right],
\left[\begin{array}{cc}
1&1\\ 1&1\\  1&0\\ 0&1\\  \end{array}\right]\}.
    \]
    Now, $\{I_2\times I_2,T_2\times T_2\}\subset\mathcal{F}''$ (the first and third matrices above), thus 
    \[
    \forb(n, \mathcal{F}'')\le\forb(n, \{I_2\times I_2,T_2\times T_2\}).
    \]
    This latter one was studied in a paper of Anstee et.al. \cite{anstee2014forbidden}, where it was shown that 
    \begin{equation}\label{eq:i2t2}
        \forb(n, \{I_2\times I_2,T_2\times T_2\})=\Theta(n^{\frac{3}{2}}).
    \end{equation}
   Applying Proposition~\ref{prop:ind} to (\ref{eq:i2t2}) twice, the proof of Proposition~\ref{prop:r=4} is complete.
\end{proof}
\begin{proof}[Proof of Theorem~\ref{thmrc4}]
    The lower bound is given by Theorem~\ref{thm:Mubayi}. The upper bound follows by induction on $r$. The base case $r=4$ is Proposition~\ref{prop:r=4}, while the induction step is Proposition~\ref{prop:ind}.
\end{proof}
\section{Remark}
	Motivated by the example and argument above, we propose the following conjecture.
	\begin{defn}
		Let $S$ be a simple matrix and $F$ be a $(0,1)-$matrix, define
		$$\Av(S,F)=\left\{A\,|\,A \hbox{ is obtained from } S \hbox{ by deleting zero or more columns}, F\not\prec A\right\},$$
		$$\forb(S,F)=\max_{A\in\Av(S,F)}\ncols{A}.$$
	\end{defn}
        
	\begin{conj}
		The exists $a,b,c\in \mathbb{Z}^{\ge 0}$, s.t. let $p=a+b+c$, 
		$$U^{a,b,c}_{n/p}=\underbrace{I_{n/p}\times\dots\times I_{n/p}}_\text{a}\times\underbrace{I^c_{n/p}\times\dots\times I^c_{n/p}}_\text{b}\times\underbrace{T_{n/p}\times\dots\times T_{n/p}}_\text{c},$$
		we have
		$$\forb(n,F)=\Theta(\forb(U^{a,b,c}_{n/p},F)).$$
	\end{conj}
\section{AI-assistant disclosure}
	
	The author of \cite{wu2026counterexample} used GPT-5.6 Sol to help identify the example and to assist with
	the exposition. The author independently reviewed and revised the outputs,
	verified the mathematical claims, and takes full responsibility for the final
	manuscript.

\end{document}